\documentclass[12pt,a4paper]{amsart}

\usepackage[utf8]{inputenc}
\usepackage{amssymb}

\newtheorem {Th}{Theorem}
\newtheorem {Le}[Th]{Lemma}
\newtheorem {Pro}[Th]{Proposition}

\theoremstyle{remark}
\newtheorem*{Rem}{Remarks}
\newtheorem*{Rem1}{Remark}

\theoremstyle{definition}
\newtheorem{Def}[Th]{Definition}

\newcommand\N{\mathbb N}
\newcommand\Q{\mathbb Q}
\newcommand\R{\mathbb R}
\newcommand\Z{\mathbb Z}
\newcommand{\NSpd}{\hspace{-3pt}}
\newcommand\mc{\mathcal}
\newcommand\eps{\varepsilon}
\def\phi{\varphi}

\begin{document}

\title{Atoms of root-closed submonoids of $\Z^2$}

\author{Günter Lettl}

\address{\hspace{-11pt}Institut für Mathematik und wissenschaftliches Rechnen\newline
Karl-Franzens-Universität\newline 
Heinrichstraße 36\newline 
A-8010  Graz, AUSTRIA}
\email{guenter.lettl@uni-graz.at}

\subjclass[2000]{Primary 20M13, secondary 13A05, 20M14, 13F15}
\keywords{root-closed monoid, factorization into atoms, continued fraction, convergent}
\date{April 27, 2022}

\begin{abstract}
We describe how one can explicitly obtain all atoms of an arbitrary root-closed monoid, whose quotient group is isomorphic to $\Z^2$. For this purpose, we solve this task for three special types of such monoids in Theorems \ref{Th5} and \ref{Th6}, and then transfer these results to the general case. It turns out that all atoms can be obtained from the (regular) continued fraction expansion of the slopes of the bounding rays of the cone, which is spanned by the monoid.
\end{abstract}

\maketitle

\section{Introduction}

Factorization theory has a long and rich history. In the course of time one realized that many factorization problems in integral domains only depend on the multiplicative structure, so many investigations and problems were translated into the theory of monoids, i.e. commutative, cancellative semigroups with a neutral element. For an overview of the relevant literature, we refer the reader to books and conference proceedings, e.g. \cite{N}, \cite{A}, \cite{Ch}, \cite{G-HK}, \cite{F-F-G-T-Z}, \cite{F-F-G-O}. 

Additive submonoids of $\Z^k$ found special attention since they have applications and connections to various areas of mathematics. In particular, they are of crucial importance in $K$-theory \cite{B-G} and toric geometry \cite{C-L-S}. 

Throughout this paper, $H$ denotes a monoid and $\mc Q (H)$ its quotient group. We will use the notations and definitions as given in \cite{G-HK}, but use addition for the operation of $H$. Let $0_H$ denote the neutral element of $H$, and $H \subset \mc Q(H) = \{h_1 - h_2 \mid h_1, h_2 \in H\}$. Furthermore, let $\mc A (H)$ denote the set of atoms of $H$. This set is fundamental to study the arithmetic and many factorization properties of $H$. The monoid $H$ is called \textit{root-closed}, if for any $x \in \mc Q(H)$ and any $n \in \N$ we have that $nx \in H$ implies $x \in H$.

The arithmetic of additive monoids with finitely generated quotient group attracted a lot of research interest in recent time, see for example \cite{F-P-U}, \cite{DA-GS-M-T}, \cite{C-F-P-U}, \cite{GS-ON-W}, \cite{Go1}, \cite{Go2}, \cite{GB-GM-MC}.
Root-closed monoids with rank at most two were investigated in \cite{C-O}, and the arithmetic of strongly primary monoids was studied in \cite{G-G-T}. Finally, it was shown in \cite{L-O-W} that the theory of root-closed monoids is axiomatizable, but not finitely axiomatizable.
\medskip

Our main interest will be in monoids $H$ with $\mc Q (H) \simeq \Z^2$, so we will tacitly assume that $H$ has rank $2$ and $H \subset \Z^2$, and put $0_H = O = (0,0)$. In this paper we will explicitly describe how one obtains the set of atoms $\mc A (H)$, supposing that $H$ is root-closed. It turns out that the atoms are closely related to the convergents and second convergents of the (regular) continued fraction expansions of the slopes of the rays bounding the cone of $H$. For this purpose we will recall results and notations from the theory of continued fraction expansions of real numbers in Section 2. 

In Section 3 we will prove the main result for special monoids, which consist of all elements of $\N_0^2$, which lie between the $x$-axis and a line with positive slope $\alpha$. It turns out that the atoms of such monoids can be determined from the continued fraction expansion of $\alpha$. The monoids $M_{\alpha, >0}^\circ$ (as defined in Definition \ref{D4}) show up e.g. in Example 4.7 in \cite{G-H}. The monoids $M_{5/2}$ and $M_5$ appear in Propositions 6.3 and 6.4 in \cite{B-R}. 
The monoids $M_\alpha$ with $\alpha = (e \pm \sqrt{e^2 -4})/2$, $3 \le e \in \N$, are tightly related to the monoids in Appendix B of \cite{A-G-W}, also referred to in Example 3.4 in \cite{Go2}. The atoms of those monoids were determined with some ingenious method, but this cannot be applied to arbitrary $\alpha$, as can be seen from our present results.

Finally, in Section 4 we will use the results of Section 3 to describe the atoms of any root-closed monoid $H \subset \Z^2$ with $\mc Q (H) = \Z^2$ in Theorem \ref{Th10}. This answers an open problem posed by F. Gotti, Question 3.5 in \cite{Go2}. 
\bigskip

\section{Notations and results for continued fractions}

For more detailed information and proofs about continued fraction expansions the reader may consult \cite{Bor} or \cite{Kh}. 

Each irrational $\alpha \in \R$ has a unique (regular) continued fraction expansion\newline $\alpha = [a_0; a_1, a_2, \dots]$ with $a_0 \in \Z$ and $a_i \in \N = \{1,2,3,\dots\}$ for all $i \in \N$. Each rational $\alpha \in \Q$ has a unique continued fraction expansion $\alpha = [a_0; a_1, a_2, \dots, a_N]$ with \textit{even} $N \in \N_0 = \N \cup \{0\}$, $a_0 \in \Z$ and $a_i \in \N$ for all $1 \le i \le N$. In the following we will tacitly assume that if $\alpha$ is rational the indices $n$ will always be bounded above such that all appearing partial quotients $a_n$ are defined.

For given $\alpha \in \R$ with continued fraction expansion as above, the numbers $p_n$, $q_n$ (with $n \le N$, if $\alpha$ is rational) are given recursively by
\[p_{-2}=0,\ p_{-1}=1,\ p_n=a_np_{n-1}+p_{n-2} \in \Z, \vspace{-6pt}
\]
\[q_{-2}=1,\ q_{-1}=0,\ q_n=a_nq_{n-1}+q_{n-2} \in \N\ \text{ for } n \ge 0.
\]
Then the following relations hold:
\begin{equation}\label{eq1}
\text{for } 0 \le n\NSpd :\ \ [a_0;a_1, \dots, a_n] = \dfrac {p_n} {q_n},
\end{equation}
\begin{equation}\label{eq2}
\text{for } -1 \le n\NSpd :\ \ p_n q_{n-1} - p_{n-1} q_n = (-1)^{n-1}\quad
\text{and for } 1 \le n\NSpd :\ \ \dfrac {p_n} {q_n} - \dfrac {p_{n-1}}{q_{n-1}} = \dfrac {(-1)^{n-1}} {q_n q_{n-1}}\,,
\end{equation}
\begin{equation}\label{eq3}
\text{for } 0 \le n\NSpd :\ \ p_n q_{n-2} - p_{n-2} q_n = (-1)^n a_n\quad
\text{and for } 2 \le n\NSpd :\ \ \dfrac {p_n} {q_n} - \dfrac {p_{n-2}}{q_{n-2}} = \dfrac {(-1)^{n} a_n} {q_n q_{n-2}}\,.
\end{equation}
The rational number $p_n/q_n$ is called the \textit{$n$-th convergent} to $\alpha$.
\medskip

\begin{Def}\label{D1}
With the above notations we define for $-2 \le n$ and\footnote{For $n=-2$ and negative $\alpha$ take $a_0 \le i \le 0$.} $0 \le i \le a_{n+2}$
\begin{equation}\label{eq4}
p_{n,i} = p_n + ip_{n+1}\quad \text{and}\quad q_{n,i} = q_n + iq_{n+1}\ .
\end{equation}
For $1 \le i \le a_{n+2}-1$ the rational numbers $\dfrac {p_{n,i}} {q_{n,i}}$ are called the \textit{second convergents} to $\alpha$. (Note that $q_{n,i}=0$ if and only if $(n,i)=(-1,0)$.)
\end{Def}
\smallskip

\begin{Rem1}
The second convergents (in German: ``Nebennäherungsbrüche") are called by Khinchin \cite{Kh} ``intermediate fractions".
\end{Rem1}
\medskip

We will often make use of the following well-known lemma (see e.g. Lemma 1.41 in \cite{Bor}):
\begin{Le}\label{L2}
Let $x_1, x_2 \in \N$ and $y_1, y_2 \in \Z$. Then 
\[ \frac{y_1}{x_1} < \frac{y_2}{x_2}\ \Longrightarrow\ \frac{y_1}{x_1} <  \frac{y_1+y_2}{x_1+x_2} < \frac{y_2}{x_2}.
\]
Obviously, ``$\le$" in the left inequality implies twice ``$\le$" in the right hand side of the above implication.
\end{Le}
\bigskip

\begin{Le}\label{L3}
With the notations of Definition \ref{D1} we have:
\smallskip

\textbf{a)} for all $-2 \le n$ and $0 \le i \le a_{n+2} :$ \ \ $\gcd(p_{n,i}, q_{n,i})=1$.
\smallskip

\textbf{b)} for all even $n \ge 0$
\begin{equation}\label{eq5}
\dfrac {p_n}{q_n} = \dfrac {p_{n,0}}{q_{n,0}} < \dfrac {p_{n,1}}{q_{n,1}} < \dfrac {p_{n,2}}{q_{n,2}} < \dots < \dfrac {p_{n,a_{n+2}}}{q_{n,a_{n+2}}} = \dfrac {p_{n+2}}{q_{n+2}} \le \alpha\ .
\end{equation}
If $\alpha > 0$ we additionally have
\[0 = \dfrac {p_{-2}}{q_{-2}} = \dfrac {p_{{-2},0}}{q_{{-2},0}} < \dots < \dfrac {p_{{-2},i}}{q_{{-2},i}} = i < \dots < \dfrac {p_{{-2},a_0}}{q_{{-2},a_0}} = \dfrac {p_{0}}{q_{0}} = a_0\ .
\]
\end{Le}
\bigskip

\begin{proof}
\textbf{a)} Using Definition \ref{D1} and \eqref{eq2} yields $p_{n,i}\, q_{n+1} - q_{n,i}\, p_{n+1} = \pm 1$. 
\smallskip

\textbf{b)} For even $n \ge 0$ we have $\dfrac {p_n} {q_n} < \alpha < \dfrac {p_{n+1}} {q_{n+1}}$. Using Lemma \ref{L2} we obtain
\smallskip

\noindent $\dfrac {p_n} {q_n} < \dfrac {p_{n,1}}{q_{n,1}} < \dfrac {p_{n+1}} {q_{n+1}}$ and inductively 
$\dfrac {p_{n,i}} {q_{n,i}} < \dfrac {p_{n,i+1}}{q_{n,i+1}} < \dfrac {p_{n+1}} {q_{n+1}}$ for all $1 \le i < a_{n+2}$.
\smallskip

\noindent The case $\alpha > 0$ and $n=-2$ follows directly from the definitions.
\end{proof}
\smallskip

The above results can be found e.g. in \cite{Kh}, Ch. I.4.
\bigskip

\section{The monoid $M_{\alpha}$ and its relatives}

Now we investigate monoids which consist of all elements of $\N_0^2$, which lie between the $x$-axis and a line with positive slope $\alpha$. We will distinguish whether the points of $\N_0^2$ on the bounding rays belong to the monoid or not, and so obtain four cases. 
In Theorems \ref{Th5} and \ref{Th6} we will explicitly give the atoms of three typical types of such monoids. The results depend on whether $\alpha$ is irrational or rational, and in the latter case also whether the points on the line with slope $\alpha$ belong to the monoid or not. From these results we will deduce a description of the atoms for $M_{\alpha, >0}$ and $M_{\alpha, >0}^\circ$ in Theorem \ref{Th7}, and finally in section 4 for any root-closed monoid $H$ with $\mc Q(H)=\Z^2$.
\medskip

\begin{Def}\label{D4}
For $0 < \alpha \in \R$ we define the following special monoids
\[M_\alpha = \{(x,y) \in \Z^2 \mid 0 \le y \le \alpha x\}\,,
\]
\[M_\alpha^\circ = \{(x,y) \in \Z^2 \mid 0 \le y < \alpha x\} \cup \{O\}\,,
\]
\[M_{\alpha, >0} = \{(x,y) \in \Z^2 \mid 0 < y \le \alpha x\} \cup \{O\}\,,
\]
\[M_{\alpha, >0}^\circ = \{(x,y) \in \Z^2 \mid 0 < y < \alpha x\} \cup \{O\}\,,
\]
where $O=(0,0)$, as before.
\end{Def}

\begin{Rem1} All the monoids of Definition \ref{D4}, and more generally, any submonoid of $\N_0^d$, are FF-monoids, i.e. each element is the sum of finitely many atoms. This is Proposition~4.2 in \cite{Go1}, and the proof is very immediate.
\end{Rem1}
\medskip

The atoms of $M_\alpha$ are explicitly given in the following 
\begin{Th}\label{Th5}
Let $0 < \alpha \in \R$ and $\alpha = [a_0; a_1, a_2, \dots]$ be its continued fraction expansion. If $\alpha \in \Q$, we choose the expansion $\alpha = [a_0; a_1, a_2, \dots, a_N]$ with even $N \ge 0$.\newline 
Using the notations of the previous section, we put for all even $n \ge -2$ and $0 \le i \le a_{n+2}$ (and $n \le N-2$, if $\alpha$ is rational)
\[A_{n,i} = (q_{n,i}, p_{n,i})\,.\vspace{4pt}
\]

\textbf{a)} If $\alpha$ is irrational, then the set of atoms of $M_\alpha$ is given by
\[\mc A (M_\alpha) = \{ A_{n,i} \mid -2 \le n \text{ even and } 0 \le i \le a_{n+2} \}\,.\vspace{4pt}
\]

\textbf{b)} If $\alpha$ is rational, then $M_\alpha$ is a finitely generated Krull monoid with\newline 
$1 + a_0 + a_2 + a_4 + \dots + a_N$ atoms, and its set of atoms is given by
\[\mc A (M_\alpha) = \{ A_{n,i} \mid -2 \le n \le N-2,\ n \text{ even and } 0 \le i \le a_{n+2} \}\,.
\]
\end{Th}
\smallskip

\begin{Rem1} Note that in the above description of the atoms we have $A_{n, a_{n+2}} = A_{n+2,0}$, which is convenient for the proof. Up to this ambiguity, all elements in the above description of $\mc A (M_\alpha)$ are pairwise different, as follows immediately from Lemma \ref{L3}.b).
\end{Rem1}
\bigskip

For the monoids $M_\alpha^\circ$ we have to consider only rational $\alpha$, since for irrational $\alpha$ one has $M_\alpha^\circ = M_\alpha$.

\begin{Th}\label{Th6}
Let $0 < \alpha \in \Q$ and $\alpha = [a_0, a_1, a_2, \dots, a_N]$ be its continued fraction expansion with even $N \ge 0$. Using the notations of Theorem \ref{Th5}, we put for $k \in \N$
\[A_k = A_{N-2,a_N-1} + k A_{N,0} = (q_{N-2, a_N-1} + kq_N, p_{N-2, a_N-1} + kp_N) \ \text{ and }\ \mc A' = \{A_k \mid k \in \N\}\,.
\]
Then the set of atoms of $M_\alpha^\circ$ is given by
\[\mc A (M_\alpha^\circ) = \mc A' \cup \mc A''\,,
\]
where
\[\mc A'' = \{ A_{n,i} \mid -2 \le n \le N-2, n \text{ even, } 0 \le i \le a_{n+2} \text{ and } (n,i) \ne (N-2, a_N)\} \,.
\]
\end{Th}
\bigskip

For the monoids $M_{\alpha, >0}$ and $M_{\alpha, >0}^\circ$ we use the results of the previous theorems to obtain the following description of the atoms:

\begin{Th}\label{Th7}
Let $0 < \alpha \in \R$. Then for the monoids $M_{\alpha, >0}$ and $M_{\alpha, >0}^\circ$ their sets of atoms are given by
\[\mc A (M_{\alpha, >0}) = (\mc A(M_\alpha) \setminus \{(1,0)\}) \cup \{(n,1) \mid n \in \N \text{ with } 1 \le \alpha n\}
\]
and
\[\mc A (M_{\alpha, >0}^\circ) = (\mc A(M_\alpha^\circ) \setminus \{(1,0)\}) \cup \{(n,1) \mid n \in \N \text{ with } 1 < \alpha n\}\,.
\]
\end{Th}
\bigskip

To characterize the atoms of $M_\alpha$ and $M_\alpha^\circ$ the following lemma is very useful:

\begin{Le}\label{L8}
Let $0 < \alpha \in \R$ and put $H=M_\alpha$ or $H=M_\alpha^\circ$.\newline 
Then $A=(x_0, y_0) \in H \setminus \{O\}$ is an atom of $H$ if and only if either

i) $y_0=0$ and $A=(1,0)$ or

ii) $y_0>0$ and for all $h=(x,y) \in H$ with $h \not\in \{ O, A\}$ and $0 < y \le y_0$ we have: $\dfrac yx < \dfrac {y_0}{x_0}$.
\end{Le}
\medskip

\begin{Rem}~

{\textbf 1.} Since for all $O \ne h=(x,y) \in H $ we have $\frac yx \le \alpha$, this lemma reveals the intimate connection between atoms of $H$ and ``{}best approximations of $\alpha$" (from below), cp. \cite{Kh}, p. 21.
\smallskip

{\textbf 2.} Geometrically, condition ii) means that the triangle with corners $O$, $A$ and $(y_0/\alpha, y_0)$ contains no elements of $H$ up to $O$ and $A$.
\end{Rem}
\medskip 

\begin{proof}[Proof of Lemma \ref{L8}]
Since $(1,0)$ is the only atom of $H$ contained in $H \cap (\Z \times \{0\})$, we may assume for the proof that $y_0 > 0$.

Let $A=(x_0, y_0)$ be an atom with $y_0>0$, and suppose that there exists some $h=(x,y) \in H \setminus \{O,A\}$ with $0 < y \le y_0$  and $\frac yx \ge \frac {y_0}{x_0}$. Then also $x_0 \ge x$, and we put $x'=x_0-x \ge 0$, $y'=y_0-y \ge 0$ and $h'=(x',y')$. Since $h \ne A$, we have $h' \ne O$. From
\begin{equation}\label{eq6} A = h + h' 
\end{equation}
we obtain with Lemma \ref{L2} \ $\frac {y'}{x'} \le \frac {y_0}{x_0} \le \frac yx \le \alpha$, so $O \ne h' \in H$, and \eqref{eq6}  is a contradiction that $A$ is an atom.
\smallskip

Now suppose that $A$ is not an atom. Then we have $h=(x,y)$, $h'=(x',y') \in H \setminus \{O\}$ with $A=h+h'$. Without loss of generality we may assume that $\frac {y'}{x'} \le \frac yx$. With Lemma \ref{L2} we obtain $\frac {y'}{x'} \le \frac {y_0}{x_0} \le \frac yx$. Since $y \le y_0$, $h$ contradicts ii) in the statement of the lemma.
\end{proof}
\bigskip

\begin{proof}[Proof of Theorem \ref{Th5}]
We will give a joint proof for {\textbf a)} and {\textbf b)}.\newline 
For $0 \le i \le a_0$ we have $A_{-2,i}=(1,i)$, and these are exactly all atoms of $M_\alpha$ with $x \le 1$.
\smallskip

Let us first prove that all indicated $A_{n,i}$ are indeed atoms of $M_\alpha$. Choose any even $n \ge 0$ and $0 \le i \le a_{n+2}$. From \eqref{eq5} we see that $A_{n,i} \in M_\alpha$. Using Lemma \ref{L8}.ii), let us suppose that there exists some $h=(x,y) \in H \setminus \{O, A_{n,i}\}$ with $0 < y \le p_{n,i}$ and $\dfrac {p_{n,i}} {q_{n,i}} \le \dfrac yx$.\newline 
By Lemma \ref{L3}.a), and since $y \le p_{n,i}$, $\dfrac {p_{n,i}} {q_{n,i}} = \dfrac yx$ would imply $h=A_{n,i}$, so we have $\dfrac {p_{n,i}} {q_{n,i}} < \dfrac yx$, and consequently $q_{n,i} > x$. Since $n$ is even (and $n \le N-2$ in case {\textbf b)}), we have\newline $\dfrac {p_{n,i}} {q_{n,i}} < \dfrac yx \le \alpha < \dfrac {p_{n+1}} {q_{n+1}}$, and obtain 
\[\dfrac 1 {q_{n+1} q_{n,i}} = \dfrac {p_{n+1}} {q_{n+1}} - \dfrac {p_{n,i}} {q_{n,i}} > \dfrac {p_{n+1}} {q_{n+1}} - \dfrac yx > 0\ .
\]
Therefore $\dfrac 1 {q_{n+1} q_{n,i}} > \dfrac 1 {x q_{n+1}}$ and $x > q_{n,i}$, a contradiction to the inequality obtained above.\medskip

Now let us show that $M_\alpha$ has no other atoms than those indicated in the theorem.\newline 
First observe that in case {\textbf b) } $A_{N-2,a_N} = (q_N, p_N)$ is an atom of $M_\alpha$ (with $\frac {p_N} {q_N} = \alpha$), and Lemma \ref{L8} shows that there exist no atoms with second component exceeding $p_N$.\smallskip

\noindent Let $A=(x_0,y_0)$ be any atom of $M_\alpha$ with $y_0 \ge 1$. We can find some $n \ge -2$ and $0 \le i < a_{n+2}$ with
\begin{equation}\label{eq7} p_{n,i} < y_0 \le p_{n,i+1}\ .
\end{equation}
(Here we use that $p_{n, a_{n+2}} = p_{n+2, 0}\,$.) Assume that $A \ne A_{n,i+1}$,  so $A_{n,i}$, $A$ and $A_{n,i+1}$ are three different atoms, and with Lemma \ref{L8}.ii) we obtain
\begin{equation}\label{eq8} \dfrac {p_{n,i}} {q_{n,i}} < \dfrac {y_0}{x_0} < \dfrac {p_{n,i+1}} {q_{n,i+1}}\ .
\end{equation}
We have $\dfrac 1 {q_{n,i+1} q_{n,i}} = \dfrac {p_{n,i+1}} {q_{n,i+1}} - \dfrac {p_{n,i}} {q_{n,i}} = \Bigl( \dfrac {p_{n,i+1}} {q_{n,i+1}} - \dfrac {y_0}{x_0} \Bigr) + \Bigl( \dfrac {y_0}{x_0} - \dfrac {p_{n,i}} {q_{n,i}} \Bigr) \ge \dfrac 1 {x_0 q_{n,i+1}} + \dfrac 1 {x_0 q_{n,i}}$, which yields $x_0 \ge q_{n,i} + q_{n,i+1}$. 
Combining this inequality with \eqref{eq7} and \eqref{eq8} we obtain
\[(q_{n,i} + q_{n, i+1}) p_{n,i} \le x_0 p_{n,i} < y_0 q_{n,i} \le q_{n,i} p_{n,i+1}
\]
and finally
\[p_{n,i}q_{n,i} < q_{n,i} p_{n,i+1} - p_{n,i} q_{n, i+1} = 1\ .
\]
This implies $p_{n,i}=0$, i.e. $(n,i)=(0,0)$ and $a_0=0$. So $A_{n,i} = A_{0,0} = (1,0)$ and $A_{n,i+1} = A_{0,1} = (1+a_1,1)$, but then $A=(x_0,1)$ with $x_0 > a_1+1$ cannot be an atom.
\smallskip

Finally, in case {\textbf b)} the remark following the statement of Theorem \ref{Th5} immediately yields the number of atoms for rational $\alpha$. It is well known that in this case $M_\alpha$ is a Krull monoid.
\end{proof}
\bigskip

\begin{proof}[Proof of Theorem \ref{Th6}]
We have $\alpha = \dfrac {p_N}{q_N}$ and put $B=(q_N,p_N) \not\in M_\alpha^\circ$.

First we will prove the following
\smallskip

\noindent\textit{Claim 1: If $A=(x_0,y_0) \in  M_\alpha^\circ$ is an atom with $y_0 \ge p_N$, then also $A-B = (x_0-q_N, y_0-p_N)$ is an atom of $ M_\alpha^\circ$.}
\smallskip

Since $\alpha(x_0-q_N) > y_0-p_N \ge 0$, we have $A-B \in M_\alpha^\circ \setminus \{O\}$.\newline Suppose that $A-B$ is not an atom.

\textit{Case 1:} Assume that $y_0-p_N=0$.\newline
Since $(1,0)$ is an atom, we have $x_0-q_N \ge 2$ and $A = (q_N+1, p_N) + (x_0-q_N-1,0)$ contradicts that A is an atom.

\textit{Case 2:} Assume that $y_0-p_N \ge 1$.\newline
By Lemma \ref{L8}.ii) there exists $h=(x,y) \in M_\alpha^\circ \setminus \{O, A-B\}$ with $0<y \le y_0-p_N$ and 
\begin{equation}\label{eq9} \dfrac yx \ge \dfrac {y_0-p_N}{x_0-q_N}\,.
\end{equation}
From this one easily obtains $0 < x \le x_0-q_N$. Using Lemma \ref{L2} and \eqref{eq9} we obtain
\[\dfrac {p_N} {q_N} > \dfrac yx \ge \dfrac {y_0-p_N} {x_0-q_N} \ge \dfrac {y_0-p_N-y} {x_0-q_N-x}\,,
\]
which shows that $O \ne A-B-h \in M_\alpha^\circ$. Again employing Lemma \ref{L2} we get
\begin{equation}\label{eq10} \dfrac {p_N} {q_N} > \dfrac {y_0-y}{x_0-x} > \dfrac {y_0-p_N-y} {x_0-q_N-x}\,.
\end{equation}
From \eqref{eq9} and \eqref{eq10} we obtain
\[x_0y - yq_N \ge xy_0 - xp_N \quad\text{and}\quad x_0p_N - xp_N > y_0q_N - yq_N\,.
\]
Adding these inequalities, we get $x_0(y+p_N) > y_0(x+q_N)$ and $\dfrac {y+p_N}{x+q_N} > \dfrac {y_0}{x_0}$. So the fact that the nonzero element $h+B$ properly divides $A$ in $M_\alpha^\circ$ contradicts that $A$ is an atom, which concludes the proof of \textit{Claim 1}.
\medskip

\noindent\textit{Claim 2: For all even $-2 \le n \le N-2$ and all $0 \le i < a_{n+2}$ we have
\[\dfrac {p_N+p_{n,i+1}} {q_N+q_{n,i+1}} > \dfrac {p_N+p_{n,i}} {q_N+q_{n,i}}\ .
\]}
\smallskip

The stated inequality is equivalent to 
$p_{n,i+1} q_N + p_N q_{n,i} + p_{n,i+1} q_{n,i} > p_N q_{n,i+1} + p_{n,i} q_N + p_{n,i} q_{n,i+1}$, which yields
$(q_N p_{n+1} - p_N q_{n+1}) + (p_{n,i+1} q_{n,i} - q_{n,i+1} p_{n,i}) > 0$.\newline 
Since $n+1 \le N-1$ is odd, $\dfrac {p_N}{q_N} = \alpha < \dfrac {p_{n+1}} {q_{n+1}}$, and furthermore $\dfrac {p_{n,i}} {q_{n,i}} < \dfrac {p_{n,i+1}} {q_{n,i+1}}$, both brackets above are positive, which proves \textit{Claim 2}.
\bigskip

Since the monoids $M_\alpha$ and $M_\alpha^\circ$ contain the same elements $(x,y)$ with $y < p_N$, it is immediate that the atoms within this subset coincide, and so by Theorem \ref{Th5} \ $\mc A''$ is exactly the set of all atoms  $(x,y) \in M_\alpha^\circ$ with $y < p_N$.
\medskip

Now we search for atoms $A=(x_0,y_0) \in  M_\alpha^\circ$ with $p_N \le y_0 < 2p_N$. By \textit{Claim 1} they must be of the form $A = B + A_{n,i}$ with $A_{n,i} \in \mc A''$. We have
\[A_{N-2, a_N-1} = (q_N - q_{N-1}, p_N - p_{N-1})\,.
\]
\textit{Case 1:} Assume that $a_N \ge 2$. One calculates $A_{N-2, a_N-2} + B = 2 A_{N-2, a_N-1}$, so this is not an atom and $\dfrac {p_N - p_{N-1}} {q_N - q_{N-1}} = \dfrac {p_N + p_{N-2, a_N-2}} {q_N + q_{N-2, a_N-2}}$. Now \textit{Claim 2} and Lemma \ref{L8}.ii) show that the only element of the form $B + A_{n,i}$ which could be an atom, is $A_1 = B + A_{N-2, a_N-1}$.
\smallskip

\noindent \textit{Case 2:} Assume that $a_N = 1$. Now we calculate $A_{N-4, a_{N-2}-1} + B = (a_{N-1}+2) A_{N-2, 0}$, so 
$\dfrac {p_{N-2}} {q_{N-2}} = \dfrac {p_N + p_{N-4, a_{N-2}-1}} {q_N + q_{N-4, a_{N-2}-1}}$, and the same ideas as above again show that the only element of the form $B + A_{n,i}$ which could be an atom, is $A_1 = B + A_{N-2, a_N-1}$.
\smallskip

Since $M_\alpha^\circ$ must have infinitely many atoms, \textit{Claim 1} proves with an inductive argument that $A_1$ must indeed be an atom, and that $A_k$ is the only atom of $M_\alpha^\circ$ with second component between $k p_N$ and $(k+1)p_N$.
\end{proof}
\bigskip

\begin{proof}[Proof of Theorem \ref{Th7}]
Since $M_{\alpha, >0} \subset M_\alpha$, all atoms of $M_\alpha$ which lie in $M_{\alpha, >0}$ are also atoms of $M_{\alpha, >0}$. Obviously, all elements $(x,1) \in M_{\alpha, >0}$ are atoms, too.\newline 
Now let us prove that $M_{\alpha, >0}$ has no further atoms. Suppose that $A=(x_0,y_0) \in M_{\alpha, >0}$ is an atom with $y_0 \ge 2$, but not an atom of $M_\alpha$. Then we have $O \ne h, h' \in M_\alpha$ with $A=h+h'$. Without restriction, we can suppose that $h=(1,0)$, and $h'=(x_0-1, y_0) \in M_\alpha$, so
\begin{equation}\label{eq11} \frac {y_0} {x_0 -1} \le \alpha\,. 
\end{equation}
Choose $n \in \N$ minimal with $\frac 1n \le \alpha$. We want to show that $A = (n,1) + (x_0-n, y_0-1)$ is a sum of 2 elements of $M_{\alpha, >0}$, a contradiction that $A$ is an atom.\newline
By our choice of $n$, \ $(n,1) \in M_{\alpha, >0}\,$. Now suppose that $\frac {y_0 -1}{x_0 -n} > \alpha$. If $n \ge 2$, we have $\frac 1{n-1} > \alpha$, and Lemma \ref{L2} yields $\frac {1 + (y_0 -1)}{(n-1) + (x_0 -n)} = \frac {y_0}{x_0 -1} > \alpha$, a contradiction to \eqref{eq11}. If $n=1$, we directly obtain $\frac {y_0 -1}{x_0 -1} < \frac {y_0}{x_0 -1} \le \alpha$, so  again $(x_0 -n, y_0 -1) \in M_{\alpha, >0}$.
\smallskip

The proof for $M_{\alpha, >0}^\circ$ is completely analogous.
\end{proof}
\bigskip

\section{Arbitrary root-closed submonoids of $\Z^2$}

Suppose for the moment that $H$ is a monoid with quotient group $\mc Q (H) = \Z^s \subset \R^s$ for some $s \in \N$, and put
\[\text{\rm cone}(H) = \Bigl\{ \sum_{i=1}^n \lambda_i h_i \mid n \in \N, h_i \in H, 0 \le \lambda_i \in \R \Bigr\}, 
\]
the convex cone (with apex $O$) spanned by $H$. Then several properties of $H$ can be characterized in geometric terms (see e.g. \cite{L}, \cite{G-HK-L}, \cite{K-L}, \cite{Gr}). We summarize some of these results in the following proposition. For the definition of those properties, which are not explicitly needed in the present paper, the reader is referred to \cite{G-HK}.

\begin{Pro}\label{Pr9}
Let $s \in \N$ and $H$ be a monoid with quotient group $\mc Q (H) = \Z^s \subset \R^s$. Then we have:
\smallskip

\textbf{a)} $H$ is root-closed if and only if $H = \text{\rm cone}(H) \cap \Z^s$.
\smallskip

\textbf{b)} $H$ is completely integrally closed if and only if $H = \overline{\text{\rm cone}(H)} \cap \Z^s$ (the bar denotes the topological closure).
\smallskip

\textbf{c)} $H$ is a Krull monoid if and only if $\text{\rm cone}(H)$ is a closed polytopal cone.
\smallskip

\textbf{d)} $H$ is primary and reduced if and only if $\text{\rm cone}(H) \setminus \{O\}$ is open and $\text{\rm cone}(H) \ne \R^s$.
\end{Pro}
\bigskip

Now let us return to the case $s=2$. Then Proposition \ref{Pr9}.\textbf{a)} shows that any root-closed monoid $H$ with $\mc Q (H) = \Z^2$ can be obtained from some convex cone $\mc C \subset \R^2$ with apex $O$, which spans $\R^2$, by $H = \mc C \cap \Z^2$. To make the results about the atoms of $H$ more clearly and avoid an abundance of possible cases, we will use automorphisms $\phi$ of the group $\Z^2$ to transform $\mc C$ into special positions. Note that then the monoids $H$ and $\phi (H)$ are isomorphic to each other, and we have $\mc A(\phi (H)) = \phi(\mc A(H))$.

As is well known, any automorphism $\phi$ of the group $\Z^2$ is given by $\phi: \binom{x}{y} \mapsto A \binom{x}{y}$ with some $A = \binom{a\ b} {c\ d} \in \text{GL}_2 (\Z)$, and we will tacitly extend this to a linear map of $\R^2$, again called $\phi$. In particular, let $\sigma_x$ denote the reflection at the $x$-axis, given by $A= \binom{1\ \ 0\ }{0\ -1}$.
\smallskip

Our main result is the following

\begin{Th}\label{Th10}
Let $H \subset \Z^2 \subset \R^2$ be a root-closed monoid with quotient group $\mc Q (H) = \Z^2$, and let $\mc C \subset \R^2$ be a convex cone\footnote{By Proposition \ref{Pr9}.a) such $\mc C$ always exists, take e.g. $\mc C = \text{\rm cone}(H)$.} with apex $O$ and $H= \mc C \cap \Z^2$. Then $H$ is isomorphic to one of the monoids $H'$ of the following list, depending on the geometric form of $\mc C$:
\medskip

\noindent\textbf{A. } $\mc C = \R^2$. \ Then $H=H'=\Z^2$ is a group.
\medskip

\noindent\textbf{B. } $\mc C$ is a halfspace of $\R^2$.
\smallskip

\textbf{B1.} $H' = \Z \times \N_0$: this is a Krull monoid with divisor theory (even a factorial monoid). Its group of units is $\Z \times \{0\}$, and there is one class of associated atoms, namely $\Z \times \{1\}$.\smallskip

\textbf{B2.} $H' = (\Z \times \N) \cup (\N_0 \times \{0\})$: one has $\mc A(H') = \{(1,0)\}$, and no $(x,y) \in H'$ with $y>0$ is an atom or a sum of atoms.
\smallskip

\textbf{B3.} $H' = (\Z \times \N) \cup \{(0,0)\}$: this is a primary BF-monoid, but not an FF-monoid. It is half-factorial, and one has $\mc A(H') = \{(n,1) \mid n \in \Z\}$.
\smallskip

\textbf{B4.} $H' = \{(x,y) \in \Z^2 \mid y \le \alpha x\}$ with some $0 < \alpha \in \R \setminus \Q$: one has $\mc A(H') = \{ \}$.
\medskip

\noindent\textbf{C. } $\mc C$ is bounded by two rays with a positive angle less than $\pi$.
\smallskip

\textbf{C1.} $H' \in \{M_\alpha, M_\alpha^\circ, M_{\alpha, >0}, M_{\alpha, >0}^\circ\}$ with some $0 < \alpha \in \R$: the properties and atoms of these monoids are given in Theorems \ref{Th5}, \ref{Th6} and \ref{Th7}.
\smallskip

\textbf{C2.} $H' = H_1 \cup \sigma_x(H_2)$, where $H_1 \in \{M_\alpha, M_\alpha^\circ\}$ and $H_2 \in \{M_\beta, M_\beta^\circ\}$
with some $0 < \alpha, \beta \in \R$: then $\mc A (H') = \mc A (H_1) \cup \sigma_x(\mc A (H_2))$, and the properties of $H'$ depend on $H_1$ and $H_2$.
\end{Th}
\smallskip

\begin{Rem1} 
The fact that $\text{\rm cone}(H) \subset \R^s$ is bounded by (at most) two rays for dimension $s=2$ makes the characterization result of Theorem \ref{Th10} feasible. A generalization to dimensions $s \ge 3$ requires completely different ideas and methods. 
\end{Rem1}
\medskip

\begin{proof} 
Let $\mc C \subset \R^2$ be any convex cone with apex $O$, which spans $\R^2$, and put $H = \mc C \cap \Z^2$.\smallskip 

\textbf{A. } If $\mc C = \R^2$, we have $H=\Z^2$, which is a group.
\medskip

\textbf{B. } Suppose that $\mc C$ is a halfspace of $\R^2$. These monoids are studied in great detail in Section 5 of \cite{C-O}.\smallskip

If the boundary of $\mc C$ is a line defined over $\Q$, it is given by $\{(x,y) \in \R^2 \mid qx-py=0\}$, where $(p,q) \in \Z^2 \setminus \{O\}$ with $\gcd (p,q) =1$. So one can find $(a,b)\in \Z^2$ with $pb + qa = 1$. 
If $(-a,b) \in H$, we consider the automorphism $\phi$ of $\Z^2$ given by $A = \binom{\ b\ a} {-q\ p}$, otherwise we have $(a,-b) \in H$ and define $\phi$ by $A = \binom{b\ \ a} {q\ -p}$. In any case, $\phi (\mc C)$ is the upper half-plane, and depending whether both, one or none bounding rays belong to $\mc C$ (and eventually use the reflection at the $y$-axis), we conclude that $H$ is isomorphic to (exactly) one of monoids $H'$ as given in \textbf{B1 - B3}.\newline
The properties of these monoids are well known, see e.g. \cite{C-O}, or the reader may check them easily. Let us just state, that in case \textbf{B3} any $(x,y)\in H$ with $y>0$ has infinitely many different representations as a sum of atoms, but each consists of exactly $y$ summands (so $H$ is half-factorial).
\smallskip
\pagebreak

Now suppose that the boundary of $\mc C$ is an irrational line (i.e. is not defined over $\Q$). If $\mc C$ is not contained in 
$\{(x,y) \in \Z^2 \mid y \le \alpha x\}$ for some $0 < \alpha \in \R \setminus \Q$, we apply $i$ times the rotation $\tau$, given by $\binom{0\ -1} {1\ \ 0}$, with some suitable $i \in \{1,2,3\}$ to obtain $H'$ as given in \textbf{B4}.\newline 
Now let $H'$ be given as in \textbf{B4}, and we will prove that $H'$ has no atoms. Take any $O \ne h = (x,y) \in H'$, so $y < \alpha x$. By Lemma \ref{L3} and \eqref{eq2} we have for all convergents $\frac {p_n}{q_n}$ of $\alpha$ with even $n$
\[0 < \alpha q_n - p_n < \frac 1 {q_{n+1}}\,,
\]
so we can find some even $n$ with $0 < \alpha q_n - p_n < \frac 1 {q_{n+1}} < \alpha x - y$. One can easily check that $h' = (q_n, p_n)$ and $h''=h-h'$ are both non-zero elements of $H'$, thus $h = h' + h''$ shows that $h$ is not an atom.
\medskip

\textbf{C. } Now suppose that $\mc C$ is bounded by two rays with a positive angle less that $\pi$.
\smallskip

If (at least) one of the rays is defined over $\Q$, we can apply the same ideas as in the first part of the proof of \textbf{B} to find an automorphism $\phi$ of $\Z^2$ such that $\phi (\mc C)$ has the positive $x$-axis as one bounding ray and the other ray lies in the upper half plane. If this other ray does not have a positive slope, we can find some $n \in \N$ with $(-n, 1) \not\in \mc C$ and apply the automorphism $\phi'$ given by $\binom{\ 1\ n+1} {0\ \ 1}$ to obtain $\phi' (\phi (H)) \in \{M_\alpha, M_\alpha^\circ, M_{\alpha, >0}, M_{\alpha, >0}^\circ\}$ with some $0 < \alpha \in \R$.
\smallskip

If (at least) one of the bounding rays of $\mc C$ is irrational, we consider the convergents to the slope of that ray. We can find some (large enough) $n \in \N$ and $\eps \in \{1,-1\}$ with $(q,p) = \eps (q_n, p_n) \not\in \mc C$, such that 
$\{(x,y) \in \R^2 \mid px-qy=0\} \cap \mc C = \{O\}$ and $(q',p') = \eps (q_{n+1}, p_{n+1})$ is an interior point of $\mc C$. Note that we have $pq' - p'q = \delta \in \{1,-1\}$.\newline 
Now we apply the automorphism $\phi$ of $\Z^2$ given by $\delta \binom{\ p\ -q} {-p'\ \ q'}$, which maps the above line, which meets $\mc C$ only in $\{O\}$, to the $y$-axis, and $(q',p')$ to $(1,0)$, which is an interior point of $\phi (\mc C)$. Thus $\phi(H)$ equals some $H'$ as given in \textbf{C2}.
\smallskip

\noindent Now let $H'$ be given as in \textbf{C2}. We finally have to show that $\mc A (H') = \mc A (H_1) \cup \sigma_x(\mc A (H_2))$. Clearly, each atom of $H'$ must be also an atom of $H_1$ or of $\sigma_x(H_2)$.\newline 
Now suppose that $(x,y) \in \mc A (H_1)$ is not an atom of $H'$. Then for $i=1,2$ we have $O \ne (x_i, y_i) \in H'$, not both in $H_1$, with $(x,y) = (x_1,y_1) + (x_2,y_2)$ and $x_1, x_2 > 0$. Supposing $y_2<0$, we have $y=y_1 + y_2 \ge 0$ and obtain $(x,y) = (x_1,y_1+y_2) + (x_2,0)$ as a sum of two elements of $H_1 \setminus \{O\}$, a contradiction.\newline 
The proof for $(x,y) \in \sigma_x (\mc A (H_2))$ is analogous.
\end{proof}
\bigskip

\begin{Rem}~

{\textbf 1.} Note that in Theorem \ref{Th10} in cases {\textbf B4}, {\textbf C1} and {\textbf C2} there are infinitely many monoids $H'$ which are isomorphic to the given $H$.\newline
Furthermore, cases {\textbf C1} and {\textbf C2} do not exclude each other, but both are necessary. If both slopes of the bounding rays of $\mc C$ are irrational, $H$ cannot be isomorphic to any $H'$ of case {\textbf C1}. Since in case {\textbf C2} any monoid $H'$ has at least $3$ atoms, $H$ cannot be isomorphic to such $H'$ if (the complete integral closure of) $H$ has only 2 atoms, i.e. is factorial. 
\smallskip

{\textbf 2.} We can also use Theorem \ref{Th10} to describe the atoms of any completely integrally closed monoid $H$ with $\mc Q(H) = \Z^2$. By Proposition \ref{Pr9}.b) we may restrict to cones $\mc C$ with $\mc C \cap \Z^2 = \overline{\mc C} \cap \Z^2$. So only the cases {\textbf A}, {\textbf B1}, {\textbf B4}, {\textbf C1} with $H'=M_\alpha$ and {\textbf C2} with $\alpha$, $\beta$ both irrational remain.
\smallskip

{\textbf 3.} Instead of using automorphisms of $\Z^2$ one could directly obtain the atoms of $H$ from the (second) convergents to the slopes of the bounding rays of $\text{\rm cone}(H)$. Then one has to modify the results of Theorems \ref{Th5} and \ref{Th6} accordingly, since one has to consider also negative slopes and to choose ``the right side" of the approximations (i.e. either even or odd indices $n$ of the (second) convergents). This would lead to an abundance of different cases, which we wanted to avoid.\newline 
Finally, it might be interesting to note that the operation of $A \in \text{GL}_2 (\Z)$ changes the slope $\alpha$ to $\frac {a+b\alpha} {c+d\alpha}$, whose continued fraction expansion differs from that of $\alpha$ only in finitely many partial quotients.
\end{Rem}
\bigskip

\textbf{Acknowledgment:} The author wants to thank Alfred Geroldinger for many discussions and also to thank the unknown referee for his/her careful reading of the manuscript. The suggestions and comments of both of them helped very much to improve the quality of this paper.

\end{document}